\documentclass[10pt,twocolumn,twoside]{IEEEtran}
\usepackage[T1]{fontenc}
\usepackage{cite}
\usepackage{amsmath,amssymb,amsfonts,mathtools}
\usepackage{graphicx}
\usepackage{xcolor}
\definecolor{DarkGreen}{RGB}{1,120,32}
\usepackage{tikz}
\usetikzlibrary{cd,arrows.meta,positioning,calc}
\usepackage[hidelinks]{hyperref}

\usepackage{amsthm}
\usepackage{bm}

\makeatletter
\def\subsection{\@startsection{subsection}{2}{\z@}
  {1.5ex plus 1ex minus 0.5ex}
  {0ex}
  {\normalfont\normalsize\itshape}}
\makeatother

\theoremstyle{plain}
\newtheorem{theorem}{Theorem}
\newtheorem{proposition}{Proposition}
\newtheorem{corollary}{Corollary}

\theoremstyle{definition}
\newtheorem{definition}{Definition}
\newtheorem{assumption}{Assumption}
\newtheorem{example}{Example}
\theoremstyle{remark}
\newtheorem{remark}{Remark}

\newcommand{\C}{\mathsf C}
\newcommand{\F}{\mathcal F}
\newcommand{\R}{\mathbb R}
\newcommand{\A}{\mathcal A}
\newcommand{\id}{\mathrm{id}}
\newcommand{\Eq}{\operatorname{Eq}}
\newcommand{\Ctrl}{\operatorname{Ctrl}}

\newcommand{\Man}{\mathsf{Man}}
\newcommand{\pr}{\operatorname{pr}}
\newcommand{\ks}{\kappa_{\rm safe}}
\newcommand{\fcl}{f_{\rm cl}}
\newcommand{\PR}{\mathbf P_{\rm RoM}}
\newcommand{\PF}{\mathbf P_{\rm FoM}}
\newcommand{\PC}{\mathbf P_{\rm cl}}
\newcommand{\MS}{\mathbf M_{\rm sf}}
\newcommand{\RH}{\mathbf R_h}
\newcommand{\RV}{\mathbf R_V}
\newcommand{\RO}{\mathbf R_0}
\newcommand{\QT}{\mathbf Q}
\newcommand{\zero}[1]{\mathbf 0_{#1}}
\newcommand{\unit}[1]{\mathbf 1_{#1}}
\newcommand{\define}[1]{\textbf{#1}}
\newcommand{\maps}{\colon}
\newcommand{\To}{\Rightarrow}
\newcommand{\into}{\hookrightarrow}

\newcommand{\bmap}[1]{\bm #1}

\tikzcdset{arrow style=tikz,diagrams={>=Stealth},column sep=large,row sep=large}

\begin{document}

\title{\LARGE \textbf{Safety of Layered Control Systems via Control Barrier Morphisms}}

\author{Aaron D. Ames, Joe Moeller, Massimiliano de Sa, and Paulo Tabuada
\thanks{The authors are with the California Institute of Technology, Pasadena CA 91125, U.S.A. {\tt\small \{ames, jmoeller, mdesa\}@caltech.edu} and UCLA, Los Angeles CA 90095-1594, {\tt\small tabuada@ee.ucla.edu}.}
\thanks{This research is funded by the AFOSR SCHEMA CoE and MURI grant Hybrid Dynamics - Deconstruction and
Aggregation (HyDDRA).}
}

\maketitle
\thispagestyle{empty}
\pagestyle{empty}

\begin{abstract}
Layered control architectures provide a powerful paradigm for synthesizing controllers for complex systems. Yet, translating guarantees at each layer to system-level guarantees is challenging due to the interactions between layers. To address this problem, this paper presents a categorical perspective on layered control architectures. Individual systems are represented as \emph{control coalgebras}, their interactions are specified as morphisms to \emph{mediators}, and system-level behaviors are obtained via composition defined categorically as a pullback. To analyze the safety of these systems, we formulate \emph{control barrier morphisms (CBMs)}, the categorical generalization of control barrier functions. We prove that feedback controllers satisfying a CBM condition render the associated safe subobject forward invariant. It is through CBMs that we can establish the safety of layered systems: given a CBM at each layer subject to compatibility conditions, the universal property of the pullback implies the existence of a CBM for the composed system. Safety of a full-order model being safety filtered by a reduced-order model is recovered as a specialization of the categorical framework. 
\end{abstract}

\section{Introduction}
\label{sec:intro}

Layered control architectures allow complex systems to be controlled through simpler models \cite{Matni2024,cohen2024safety}.
A canonical example is the ``planner-tracker'' paradigm, where a planning (or safety) layer computes a reference using a reduced-order model, while a low-level controller on the full-order model tracks this reference signal. These architectures have the advantage of decomposing control problems into simpler subproblems that are tractable, but also introduce a question that cannot be answered by analyzing either layer in isolation: when does a guarantee established on the reduced-order model remain valid on the full-order model? This question has increased in importance due to the need to establish safety guarantees for highly dynamic robotic systems.

The problem of identifying architectures has been investigated frequently over the years from a variety of perspectives, including: aggregation \cite{Aoki1968Aggregation}, hierarchies \cite{Mesarovic1970Hierarchical}, supervision \cite{ZhongWonham1990Consistency}, behavioral interconnection \cite{Willems2007Behavioral}, and symbolic abstractions \cite{Tabuada2009Verification}, to name a few.  More recent approaches formulate layered architectures through contracts \cite{Mazo2024Contract,10.1145/3764587} and multi-rate control \cite{Rosolia2022}. 
These complementary perspectives indicate the fundamental nature of control architectures, and the need to make guarantees for complex systems by studying constituent components and their interactions. This paper studies layered architectures from the perspective of safety utilizing the framework of category theory, enabling us to determine how the safety of layered control architectures can be established from the safety of individual layers and their interactions.  

As an illustration of the results in this paper, consider
Figure~\ref{fig:architecture} where  a Control Barrier Function (CBF)-based safety filter \cite{AmesCBF2017}, designed for a Reduced Order Model (RoM), is implemented in a Full Order Model (FoM) via an existing tracking controller. 
The safety filter utilizes a CBF on the RoM to transform a desired input signal into an input signal that guarantees safety.  
The FoM realizes the input via a tracking controller, introducing tracking error. 
Prior results combine the RoM barrier with tracking guarantees on the FoM to obtain a safety guarantee for the FoM~\cite{Molnar2022,Compton2025}. This paradigm provides a starting point for building a theory of composition that yields safety guarantees. 

Category theory gives a natural language in which to reason about the composition.  \emph{Coalgebras} provide a general formulation for systems and their structure-preserving maps \cite{Rutten2000}.  \emph{Pullbacks} describe heterogeneous composition of tagged systems \cite{AmesNESC2006} and nonlinear feedback interconnection \cite{Tabuada2008}.
\emph{Diagrams} allow for the interconnection of open dynamical systems  \cite{Vagner2015}, for hybrid systems to be composed \cite{Culbertson2020}, and for  the composition of Lyapunov certificates via assume-guarantee reasoning \cite{CapucciMyers2026}.  This paper explores categorical composition through the lens of safety. 

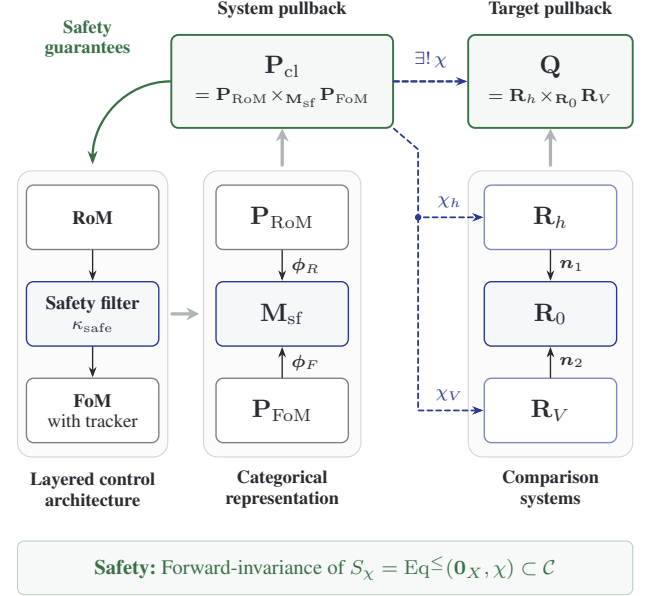
\begin{figure}[t]
\centering
\resizebox{1.0\columnwidth}{!}{
\begingroup
\definecolor{heroBlue}{RGB}{42,57,143}
\definecolor{heroGreen}{RGB}{58,112,70}
\definecolor{heroInk}{RGB}{40,44,49}
\definecolor{heroGray}{RGB}{131,139,146}
\begin{tikzpicture}[
  x=1mm,y=1mm,
  >=Stealth,
  line cap=round,line join=round,
  every node/.style={font=\fontsize{7.5}{8.3}\selectfont,
                    text=heroInk,inner sep=1.2pt,outer sep=.3pt},
  hero/panel/.style={draw=black!16,fill=black!1.5,
                    rounded corners=2mm,line width=.45pt},
  hero/box/.style={draw=black!50,fill=white,rounded corners=1mm,
                  line width=.55pt,align=center,
                  minimum width=18.5mm,minimum height=9mm},
  hero/med/.style={hero/box,draw=heroBlue,fill=heroBlue!4},
  hero/target/.style={hero/box,draw=heroBlue!60,fill=white},
  hero/composite/.style={draw=heroGreen,fill=heroGreen!5,
                        rounded corners=1mm,line width=.75pt,
                        align=center,minimum height=13mm},
  hero/strict/.style={-{Stealth[length=1.5mm,width=1.1mm]},
                     draw=heroInk,line width=.55pt},
  hero/cert/.style={-{Stealth[length=1.6mm,width=1.2mm]},
                   draw=heroBlue,dash pattern=on 2pt off 1.5pt,
                   line width=.7pt},
  hero/certline/.style={draw=heroBlue,dash pattern=on 2pt off 1.5pt,
                       line width=.7pt},
  hero/construction/.style={-{Stealth[length=1.8mm,width=1.8mm]},
                           draw=heroGray!65,line width=1.3pt},
  hero/maplabel/.style={font=\fontsize{7.4}{8}\selectfont,
                       inner sep=.5pt,text=heroInk},
  hero/title/.style={font=\fontsize{7.3}{8.1}\selectfont\bfseries,
                    align=center,inner sep=0pt},
  hero/notation/.style={font=\fontsize{10.3}{11.5}\selectfont}
]

\path[hero/panel] (1,5.5) rectangle (22,45.5);
\path[hero/panel] (27,5.5) rectangle (49,45.5);
\path[hero/panel] (64,5.5) rectangle (87,45.5);

\node[hero/box] (rom) at (11.5,39) {\textbf{RoM}};
\node[hero/med] (filter) at (11.5,25.5)
  {\textbf{Safety filter}\\[-.1mm] $\kappa_{\rm safe}$};
\node[hero/box] (fom) at (11.5,12)
  {\textbf{FoM}\\[-.1mm] with tracker};
\draw[hero/strict] (rom.south) -- (filter.north);
\draw[hero/strict] (filter.south) -- (fom.north);

\node[hero/box,hero/notation] (PR) at (38,39)
  {$\mathbf P_{\rm RoM}$};
\node[hero/med,hero/notation] (MS) at (38,25.5)
  {$\mathbf M_{\rm sf}$};
\node[hero/box,hero/notation] (PF) at (38,12)
  {$\mathbf P_{\rm FoM}$};
\draw[hero/strict] (PR.south) -- (MS.north)
  node[midway,right=1mm,hero/maplabel] {$\bmap{\phi}_R$};
\draw[hero/strict] (PF.north) -- (MS.south)
  node[midway,right=1mm,hero/maplabel] {$\bmap\phi_F$};

\node[hero/target,hero/notation] (RH) at (75.5,39)
  {$\mathbf R_h$};
\node[hero/med,hero/notation] (R0) at (75.5,25.5)
  {$\mathbf R_0$};
\node[hero/target,hero/notation] (RV) at (75.5,12)
  {$\mathbf R_V$};
\draw[hero/strict] (RH.south) -- (R0.north)
  node[midway,right=1mm,hero/maplabel] {$\bmap{n}_1$};
\draw[hero/strict] (RV.north) -- (R0.south)
  node[midway,right=1mm,hero/maplabel] {$\bmap{n}_2$};

\draw[hero/construction] (22.6,25.5) -- (26.4,25.5);

\node[hero/composite,minimum width=31mm] (PC) at (38,58)
  {{\fontsize{10.2}{11.2}\selectfont
      $\mathbf P_{\rm cl}$}\\[.9mm]
   {\fontsize{7.3}{8}\selectfont
      $=\mathbf P_{\rm RoM}\!\times_{\mathbf M_{\rm sf}}\!\mathbf P_{\rm FoM}$}};
\node[hero/composite,minimum width=23mm] (Q) at (75.5,58)
  {{\fontsize{10.2}{11.2}\selectfont $\mathbf Q$}\\[.9mm]
   {\fontsize{7.3}{8}\selectfont
      $=\mathbf R_h\!\times_{\mathbf R_0}\!\mathbf R_V$}};
\node[hero/title] at (38,68) {System pullback};
\node[hero/title] at (75.5,68) {Target pullback};
\draw[hero/construction] (38,46.3) -- (PC.south);
\draw[hero/construction] (75.5,46.3) -- (Q.south);

\draw[hero/cert,line width=.95pt] (PC.east) -- (Q.west)
  node[midway,above=1.7mm,hero/maplabel,text=heroBlue]
  {$\exists!\,\chi$};

\draw[hero/cert] (PC.south east) -- (57,48.5) -- (57,39)
  -- (RH.west);
\draw[hero/cert] (57,39) -- (57,12) -- (RV.west);
\fill[heroBlue] (57,39) circle[radius=.45mm];
\node[hero/maplabel,text=heroBlue,anchor=south] at (61.4,40.1)
  {$\chi_h$};
\node[hero/maplabel,text=heroBlue,anchor=south] at (61.4,13.1)
  {$\chi_V$};

\draw[-{Stealth[length=1.9mm,width=1.4mm]},heroGreen,
      line width=.9pt]
  (PC.west) to[out=180,in=90] (11.5,46.4);
\node[hero/title,text=heroGreen] at (10.8,64)
  {Safety\\guarantees};

\node[hero/title] at (11.5,1)
  {Layered control\\architecture};
\node[hero/title] at (38,1)
  {Categorical\\representation};
\node[hero/title] at (75.5,1)
  {Comparison\\systems};

\node[draw=heroGreen!35,fill=heroGreen!4,rounded corners=.9mm,
      minimum width=86mm,minimum height=7.5mm,align=center,
      font=\fontsize{8.2}{9}\selectfont,text=heroGreen]
  at (44,-10)
  {\textbf{Safety:} Forward-invariance of
   $S_\chi=\Eq^{\le}(\zero X,\chi) \subset\mathcal C$};

\end{tikzpicture}
\endgroup
}
\caption{
A classic layered control architecture represented categorically, for which the closed loop system is the pullback.  Each layer has an associated control barrier morphism, and the universal property of the pullback implies the existence of a control barrier morphism for the layered architecture.
\vspace{-0.4cm}
}
\label{fig:composition-overview}
\end{figure}

The first contribution of this paper is a categorical formulation of CBFs. To this end, we generalize control systems to a categorical setting by modeling them as coalgebras. Taking inspiration from categorical Lyapunov theory \cite{CLT1,CLT2}, we formulate a \emph{control barrier morphism (CBM)} as a map from a control coalgebra to a target comparison system. The collection of safe actions is represented by the \emph{admissible locus} of a CBM. The existence of a feedback controller (section) of the admissible locus yields a closed loop system that is safe: it renders a subobject defined by the CBM forward invariant. 

The second contribution is to frame composition as the categorical pullback. This endows composition with a \emph{universal property}. Given two control systems, their interaction is framed as morphisms to a \emph{mediator}, and the interconnected system is the \emph{pullback}. The universal property of the pullback together with barrier morphisms enables the main result: given CBMs on each system that agree on the mediator, the result is a CBM for the composed system implying safety of a subobject. We apply this result to the motivating problem: a CBF-safety filter utilizing a RoM and tracked by a FoM, thus demonstrating that specialization of the main result leads to a CBF guaranteeing safety of the layered architecture. Importantly, the CBF for the layered system emerges as a consequence of the categorical framework, rather than from heuristic constructions, demonstrating its value. 

\section{Motivation: Layered Safety Filters}
\label{sec:background}

Consider a layered architecture (cf. Fig. \ref{fig:architecture}) consisting of a reduced-order model (RoM) and full-order model (FoM): 
\begin{equation}
\begin{aligned}
\textrm{RoM} \qquad  \dot z & =f_{\rm R}(z,\nu):=f(z)+g(z)\nu, \\
\textrm{FoM} \qquad \dot x & =f_{\rm F}(x,v),
\end{aligned}
 \label{eq:models}
\end{equation}
where the states $z$ and $x$ range over open subsets $Z\subseteq\R^{n_z}$ and $X\subseteq\R^{n_x}$, and the inputs $\nu$ and $v$ range over $U_\nu=\R^m$ and $U_v=\R^m$, respectively. 
We keep separate names for the two input spaces, even though they coincide, because the safety filter in Section~\ref{sec:application} treats them as distinct interfaces.
The inputs to the RoM and FoM have the same dimension because we are implicitly assuming that $f_{\rm F}$ is obtained from the ``true'' dynamics of the system, $\dot{x} = F(x) + G(x) u$ together with a tracking controller $u = K(x,v)$ wherein $f_{\rm F}(x,v) = F(x) + G(x) K(x,v)$. Based on the  inputs $\nu$ of the RoM and the state $x$ of the FoM, the tracking controller forces the FoM to converge to the RoM's behavior.

The RoM and FoM are coupled by a safety filter. Given the \textit{state projection} $\Pi \colon X \to Z$ and a safety function $h \colon Z\to\R$ the safe set for the FoM is given by:   
\begin{equation}
 \mathcal C=\{x\in X:h(\Pi(x))\ge0\}.
 \label{eq:physical-safe}
\end{equation}
We assume that $h$ is a control barrier function (CBF) for the RoM, with the corresponding CBF-QP safety filter: 
\begin{align}
\label{eq:qp}
 \ks(z)= & \underset{\nu\in U_{\nu}}{\operatorname{argmin}}\quad
 \tfrac12\|\nu-\nu_{\rm des}(z)\|^2\qquad\\
  \qquad & ~\textrm{s.t.} \quad  L_fh(z)+L_gh(z)\nu\ge-\alpha h(z)
 + \tfrac{\| L_g h(z) \|^2}{4 \gamma \rho} \nonumber
\end{align}
with $\alpha> 0$ (taken to be a constant for simplicity, although a  $\mathcal{K}_{\infty}^e$ function would also be possible), $L_fh=Dh\,f$ and $L_gh=Dh\,g$ are the Lie derivatives, and the normed squared term is for robustness. In the layered architecture, the safe controller is passed to the FoM via the tracking controller, i.e., setting the reference velocity to $v = \kappa_{\rm safe}(\Pi(x))$ wherein:
\begin{equation}
 \dot x=\fcl(x):=f_{\rm F}(x,\ks(\Pi(x))).
 \label{eq:closedloop}
\end{equation}

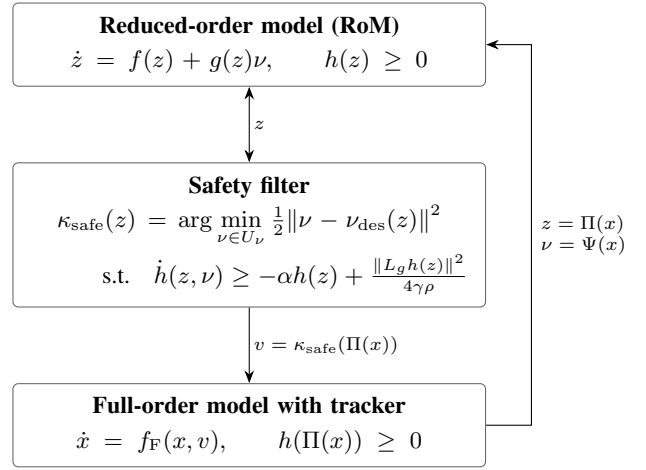
\begin{figure}[t]
\centering
\begin{tikzpicture}[>=Stealth, 
layer/.style={draw=black!60, rounded corners=2pt, align=center, text width=5.9cm, inner sep=5pt, font=\small},
lab/.style={font=\scriptsize, fill=white, inner sep=1.5pt}]
    \node[layer] (r) {\textbf{Reduced-order model (RoM)}\\[1mm]
    $\dot z=f(z)+g(z)\nu,\qquad h(z)\ge0$};
    \node[layer,below=10mm of r] (s) {\textbf{Safety filter}\\[1mm]
    $\displaystyle\ks(z)=\arg\min_{\nu\in U_{\nu}}\tfrac12\|\nu-\nu_{\rm des}(z)\|^2$\\[1mm]
    $\qquad  \quad \text{s.t.}\quad \dot{h}(z,\nu)  \ge-\alpha h(z) +  \frac{\| L_g h(z) \|^2}{4 \gamma \rho}$};
    \node[layer,below=10mm of s] (f) {\textbf{Full-order model with tracker}\\[1mm]
    $\dot x=f_{\rm F}(x,v),\qquad h(\Pi(x))\ge0$};
    \draw[<->] (r.south)--node[right,lab] {$z$}(s.north);
    \draw[->] (s.south)--node[right,lab] {$v=\ks(\Pi(x))$}(f.north);
    \draw[->] (f.east)--++(.60,0)|-node[pos=.25,right,font=\scriptsize,align=left]
    {$z=\Pi(x)$\\$\nu=\Psi(x)$}(r.east);
\end{tikzpicture}
\caption{A layered safety architecture: a reduced-order model is used to synthesize a safe controller, $\kappa_{\rm safe}$, which is fed to a full-order model as a reference $v = \kappa_{\rm safe}(\Pi(x))$ that is tracked, where $\Pi$ is the state projection. The goal is to guarantee safety on the full-order model: $h(\Pi(x)) \geq 0$. 
}
\label{fig:architecture}
\end{figure}

The goal is safety of the FoM, i.e., to keep the state of the system within $\mathcal C$. To this end, consider the FoM tracking error associated with the reference $v =  \kappa_{\rm safe}(\Pi(x))$: 
\[
e(x) =\Psi(x)-\ks(\Pi(x))
\]
where $\Psi \colon X \to U_{\nu}$ is the \textit{input projection}. That is, the error is the difference between the ``actual velocity'' of the system projected to the RoM input space via $\Psi$ and the ``desired velocity'' encoded by $\ks$.

A \textbf{tracking function} $V$ quantifies the ability of the FoM to bound this error via the assumptions: 
\begin{equation}
 \begin{aligned}
\textrm{tracking function:} \quad V(x) & \ge \rho\|e(x)\|^2,\\
\textrm{tracking dynamics:} \quad \ \quad \dot V & \le-\lambda V, \\
\textrm{barrier dynamics:} \quad \ \ \quad  \dot b &\ge-\alpha b-\gamma V. 
 \end{aligned}
 \label{eq:classical-bounds}
\end{equation}
for $b:=h\circ\Pi$ and $\gamma,\rho>0$. These assumptions can be used to define the \emph{composite barrier function} for the FoM: 
\[
 H(x):=(\lambda-\alpha)b(x)-\gamma V(x). 
\]
The main result of \cite{Molnar2022} is that if $\lambda > \alpha$: 
\begin{equation}
 H(c(0))\ge0\  \qquad \Longrightarrow\  \qquad c(t)\in\mathcal C \quad \forall t \in I
 \label{eq:classical-result}
\end{equation}
where $c \colon I = [0,t_f) \subseteq \R \to X$ is a solution: $\dot{c} = f_{\rm cl}(c(t))$. 
Therefore, starting from a reduced set of initial conditions defined by the zero superlevel set of $H$, the state remains within $\mathcal C$ and, therefore, the FoM is safe. These ideas were generalized in \cite{cohen2025safety} to bounded tracking and in \cite{Compton2025} to predictive CBFs. 
See~\cite{nanayakkara2026safetystateuncertaintyrobustifying} for a notion of robust CBF, generalizing input-to-state safety used in~\cite{Molnar2022}. 
This paper will demonstrate in Section~\ref{sec:application} that this result is a natural consequence of categorical composition theory.

\section{Control Barrier Morphisms}
\label{sec:control}

Categorical Lyapunov theory \cite{CLT1,CLT2} generalized Lyapunov theory to a broad class of ``systems'' by leveraging coalgebras. We build upon this framework and refer the reader to these papers for the background content needed. Importantly, this paper extends categorical Lyapunov theory to ``control systems'' through the introduction of \emph{control coalgebras}. We leverage this to generalize safety-critical control to a categorical setting: defining control barrier morphisms (generalizing CBFs categorically) that imply forward invariance of a subobject (via a categorical Nagumo theorem). 

Consider a category $\C$ with finite products and terminal object $1$. A point of $B$ is a map $x \colon 1\to B$, and $!_X \colon X\to1$ denotes the unique map to the terminal object. In the remainder of the paper we will use $\F \colon \C\to\C$ to denote a functor and $\pi \colon \F\Rightarrow\id_{\C}$ to denote a natural transformation from $\F$ to the identity functor on $\C$.

\subsection{Systems, inputs, and feedback}

A \textbf{coalgebra} is a morphism $f \colon B \to \F B$. This generalizes ``systems,'' e.g., a smooth vector field $\dot x =f(x)$ is a $T$-coalgebra $f \colon M \to TM$ where $T$ is the tangent bundle functor on $\C=\Man$, manifolds with boundary. 
Control coalgebras generalize this to include inputs. 

\begin{definition}[Control coalgebra]
    Let $\C$ be a category with finite products, $\F \maps \C \to \C$ be an endofunctor, and $\pi \maps \F \To \id_\C$ a natural transformation. A \define{control coalgebra}, denoted by $\mathbf{P} = (p \maps E \to B, f)$, consists of 
    \begin{enumerate}
        \renewcommand{\labelenumi}{(\roman{enumi})}
        \item a base object $B \in \C$ called the \define{state space}
        \item a bundle $p \maps E \to B$ called the \define{control bundle}
        \item a bundle map $f \maps (p\maps E \to B) \to (\pi_B \maps \F B \to B)$, i.e., a map $f \maps E \to \F B$ called the \define{control system}, such that $\pi_B \circ f = p$ as illustrated by the commutative diagram: 
    \[
    \begin{tikzcd}[row sep = 6mm]
        E
        \arrow[rr, "f"]
        \arrow[dr, "p"']
        &&
        \F B
        \arrow[dl, "\pi_B"]
        \\&
        B
    \end{tikzcd}
    \]
    \end{enumerate}

\end{definition}

A \define{morphism} of control coalgebras $\bmap{\phi} = (\phi^B,\phi^E)\colon \mathbf P\to\mathbf P'$ consists of maps $\phi^B \colon B\to B'$ and $\phi^E \colon E\to E'$ with:
\begin{equation}
 p'\circ\phi^E=\phi^B\circ p,\qquad
 f'\circ\phi^E=\F\phi^B\circ f.
 \label{eq:strict}
\end{equation}
These objects and maps form the category $\Ctrl_{\F}(\C)$.

A \define{feedback controller} is a section $k\colon B\to E$, i.e., $p\circ k=\id_B$. It produces the \emph{autonomous system} $\fcl=f\circ k \colon B\to\F B$ as illustrated by a commutative diagram: 
\[
\begin{tikzcd}[row sep = 6mm]
    E
    \arrow[rr, "f"]
    \arrow[dr, "p"description]
    &&
    \F B
    \arrow[dl, "\pi_B"description]
    \\ 
    &
    B
    \arrow["k", from=2-2, to=1-1, bend left = 4 0]
    \arrow["f_{\rm cl} := f \circ k"', from=2-2, to=1-3, bend right = 40]
\end{tikzcd}
\]
Autonomous systems are the control coalgebras with $p=\id_B$, i.e., $(\id_B \colon B \to B,f_{\rm cl})$, and therefore satisfy $\pi_B \circ \fcl=\id_B$.

\begin{example}[Continuous-time control systems]
\label{ex:continuous}
For $\C=\Man$ and $\F=T$, the tangent bundle functor, the natural projection is $\pi_M \colon TM\to M$. In the case that $E=M\times U$ and $p=p_M$, a control coalgebra is illustrated by: 
\[
\begin{tikzcd}[row sep = 6mm]
    M \times U
    \arrow[rr, "f"]
    \arrow[dr, "p"description]
    &&
    T M 
    \arrow[dl, "\pi_M"description]
    \\&
    M
    \arrow["k", from=2-2, to=1-1, bend left = 40]
    \arrow["f_{\rm cl} := f \circ k"', from=2-2, to=1-3, bend right = 40]
\end{tikzcd}
\]
In local coordinates, we recover the classical notions: 
\begin{itemize}
    \item control system: $\dot x=f(x,u)$ for $u \in U$, 
    \item feedback controller:
    $k(x)=(x,\kappa(x))$, 
    \item closed-loop system:
    $f_{\rm cl}(x)=f(x,\kappa(x))$.
\end{itemize}
\end{example}

Let $T$ be a time object in $\C$, i.e., an object equipped with a distinguished point $0_T \colon 1 \to T$. Following \cite{moeller2026hybrid}, consider a collection of \define{intervals} $I$ consisting of subobjects $\iota \colon I \into T$ with initial point $0_I \colon 1\to I$ satisfying $\iota \circ 0_I = 0_T$. On these intervals, a \define{unit clock} $\unit I \colon I\to\F I$ is a system with $\pi_I \circ \unit I=\id_I$. A \define{solution} of an autonomous system $f_{\rm cl} \colon B\to\F B$ is a map $c \colon I\to B$ satisfying\footnote{For the benefit of the reader less familiar with (lax) commuting diagrams, we use $\rightsquigarrow$ to indicate the (in)equality the diagram represents.}: 
\begin{equation}
\label{eqn:solutiondiag}
    \begin{tikzcd}
        I
        \arrow[r,"c"]
        \arrow[d,"\unit{I}"']
        &
        B
        \arrow[d,"f_{\rm cl}"]
        \\
        \F I
        \arrow[r,"{\F c}"']
        &
        \F B
    \end{tikzcd} 
    \qquad \rightsquigarrow \qquad 
    \F c\circ\unit I=f_{\rm cl}\circ c
\end{equation}
For Example~\ref{ex:continuous}, $T = \R$, $I = [0,t_f) \subseteq \R$ with $t_f > 0$, $\unit I(t)=(t,1)$ and \eqref{eqn:solutiondiag} implies that $\dot c(t)=f_{\rm cl}(c(t))$. 

\subsection{Lax equalizers, barrier morphisms and safe subobjects}

\begin{definition}[Orders]
\label{def:posetaldef}
An object $R\in\C$ is \define{preordered} if, for every object $X\in\C$, the hom-set $\C(X,R)$ is equipped with a preorder, $f_1 \leq f_2$, and precomposition is order-preserving: if $g_1\le g_2\maps Y\to R$ and $f\maps X\to Y$, then $g_1\circ f\le g_2\circ f$. 

A preordered object $R $ is \define{posetal} if the preorder on $\C(X,R)$ is antisymmetric: $f_1 \leq f_2$ and $f_2 \leq f_1$ implies $f_1 =  f_2$. 
\end{definition}

This notion of preordered objects allows for the formulation of lax equalizers defined via a \emph{universal property}.

\begin{definition}[Lax equalizers]
    Let $R$ be preordered. The \define{lax equalizer} of a pair of maps $a,b \colon X \to R$ is an object $\Eq^{\le}(a,b)$ and map $e \colon \Eq^{\le}(a,b) \to X$ satisfying: 
    \[
    \begin{tikzcd}[row sep = 6mm]
        \Eq^\leq(a,b)
        \arrow[r, "e"]
        \arrow[d, "e"']
        &
        X
        \arrow[d, "b"]
        \\
        X
        \arrow[r, "a"']
        \arrow[ur, phantom, sloped, "\leq", pos = 0.5]
        &
        R
    \end{tikzcd}
    \quad \rightsquigarrow \quad 
    a \circ e \leq b \circ e , 
    \]
    such that for every $q \colon Q\to X$,
    \begin{equation}
        a\circ q\le b\circ q
        ~ \Longleftrightarrow ~ 
        \exists!\,u \colon Q\to \Eq^\leq(a,b) \text{ with }e\circ u=q.
    \label{eq:laxeq}
    \end{equation}
\end{definition}

For a set $X$ with $a , b \colon X \to \R$, the lax equalizer is $\Eq^{\le}(a,b) = \{x \in X:a(x)\le b(x)\}$. Thus a lax commuting diagram describes an inequality, and the lax equalizer describes the locus on which it holds. 
Lax equalizers always exist in $\mathsf{Set}$ and $\mathsf{Top}$, but need not exist in $\Man$. For instance, if $0$ is a not regular value of $h$, then  $\Eq^{\le}(0,h) = \{x \in X : 0 \leq h(x)\}$ may not be an object of $\Man$, as it may have a nonsmooth boundary. Throughout, we assume that the lax equalizers we use exist in the ambient category.

Choose a posetal measurement object $R$ with distinguished point $0_R \colon 1\to R$, an autonomous comparison system $\sigma \colon R\to\F R$, a preorder on $\F R$, and write $\zero X := 0_R\circ!_X \colon X\to R$. The posetal order on $R$ and preorder on $\F R$ are distinct and related via the following essential assumption. 

\begin{assumption}[Barrier comparison lemma]
\label{ass:comparison}
The comparison system $\sigma \colon R \to \F R$, together with the order structures on $R$ and $\F R$, and the distinguished point $0_R$, satisfies the following for every time domain $I$ and map $u\colon I\to R$:
\begin{equation}
 u\circ0_I\ge0_R,\quad
 \F u\circ\unit I\ge\sigma\circ u
 \quad\Longrightarrow\quad u \ge \zero I .
 \label{eq:comparison}
\end{equation}
\end{assumption}
This assumption is the categorical version of the comparison lemma \cite{NonlinearSystems}, which is recovered when we make the choice $R = \R$ with comparison system $\dot r = \sigma(r) := - \alpha r$ for $\alpha > 0$. The classical comparison lemma gives $u(t) \geq e^{-\alpha t}u(0) \geq 0$ for any curve with $\dot u(t)\ge-\alpha u(t)$ and $u(0) \geq 0$. 

\begin{definition}[Barrier morphism and safe subobject]
For an autonomous system $f_{\rm cl}\colon B\to\F B$, a morphism $h\colon B\to R$ is a \define{barrier morphism} if:
\begin{equation}
\label{eqn:barrierdiag}
\begin{tikzcd}[row sep = 6mm]
    B
    \arrow[r, "h"]
    \arrow[d, "f_{\rm cl}"']
    &
    R
    \arrow[d, "\sigma"]
    \\
    \F B 
    \arrow[r, "\F h "']
    \arrow[ur, phantom, sloped, "\geq", pos = 0.4]
    &
    \F R 
\end{tikzcd}
 \quad \rightsquigarrow \quad 
 \F h\circ f_{\rm cl}\ge\sigma\circ h. 
\end{equation}
Its \define{safe subobject} is $s \colon S_h:=\Eq^{\le}(\zero B,h)\hookrightarrow B$.
\end{definition}

A subobject $s \colon S\hookrightarrow B$ is \define{forward invariant} if every solution $c \colon I\to B$ starting in $S$ factors through $s$. That is, $c\circ0_I=s\circ x_0$ for some $x_0 \colon 1\to S$ implies $c=s\circ\widetilde c$ for a (necessarily unique) $\widetilde c \colon I\to S$.

\begin{theorem}[Categorical Nagumo theorem]
\label{thm:barrier}
Under Assumption~\ref{ass:comparison}, a barrier morphism $h$ renders $S_h$ forward invariant.
\end{theorem}

\begin{proof}
Given a solution $c$ starting in $S_h$, pasting the solution diagram \eqref{eqn:solutiondiag} with the lax barrier diagram \eqref{eqn:barrierdiag} yields: 
\[
\begin{tikzcd}[column sep=12mm,row sep=6mm]
    I
    \ar[r,"c"]
    \ar[d,"\unit I"']
    & 
    B
    \ar[r,"h"]
    \ar[d,"f_{\rm cl}"']
    & 
    R
    \ar[d,"\sigma"]
    \\
    \F I
    \ar[r,"\F c"']
    &
    \F B 
    \ar[ur, phantom, sloped, "\geq", pos = 0.5] 
    \ar[r,"\F h"'] 
    &
    \F R.
\end{tikzcd}
\]
Then $u=h\circ c$ satisfies~\eqref{eq:comparison}, so $\zero I\le h\circ c$. The universal property~\eqref{eq:laxeq} gives the unique factorization of $c$ through $S_h$ implying that $S_h$ is forward invariant.
\end{proof}

\subsection{Control barrier morphisms}

Barrier morphisms can be extended to a control setting yielding \emph{control barrier morphisms}. 
Consider a control coalgebra $\mathbf P = (p \colon E \to B,f)$ and a target comparison system $\sigma \colon R \to \F R$ with $R$ posetal and $\F R$ preordered. 

\begin{definition}[Admissible locus]
For a  map $h \colon B\to R$, define the \define{admissible locus} $\A_h$ of $h$ as the lax equalizer: 
\begin{equation}
\begin{tikzcd}[cramped,column sep=4.6mm,row sep = 2mm]
	& {E } & B & R \\
	{\mathcal{A}_h := \Eq^\leq(\sigma \circ h \circ p, \mathcal{F} h \circ f)} \\
	& E & \F B & \F R
	\arrow["p", from=1-2, to=1-3]
	\arrow["h", from=1-3, to=1-4]
	\arrow["\sigma", from=1-4, to=3-4]
	\arrow["{a_h}", from=2-1, to=1-2]
	\arrow["{a_h}"', from=2-1, to=3-2]
	\arrow[from=3-2, to=1-4, phantom, sloped, "\geq", pos = 0.5]
	\arrow["f~"', from=3-2, to=3-3]
	\arrow["\F h~"', from=3-3, to=3-4]
\end{tikzcd}
 \label{eq:admissible}
\end{equation}
\end{definition}

\begin{definition}[Control barrier morphism]
\label{def:CBM}
A morphism $h\maps B\to R$ is a \define{control barrier morphism (CBM)} if:
    \begin{enumerate}
        \renewcommand{\labelenumi}{(\roman{enumi})}
    \item $\pi_h:=p\circ a_h \colon \A_h\longrightarrow B$ is an \emph{epimorphism}, 
    \item $\pi_h$ is \emph{point-surjective}\footnote{It is possible for epimorphisms to not be surjective on points. For example, the inclusion $\R^2\setminus\{(0,0)\} \hookrightarrow \R^2$ is an epimorphism in $\Man$. This is because two smooth functions on the plane that agree everywhere away from the origin must also agree at the origin.}: $\C(1,\pi_h) \colon \C(1,\A_h) \to \C(1,B)$ is surjective.
    \end{enumerate}
\end{definition}

Because $\pi_h$ is point-surjective, every state $x \colon 1\to B$ admits an admissible
point $\widetilde x \colon 1\to\A_h$ satisfying
$\pi_h\circ\widetilde x=x$. 
In categories with an underlying set structure, this condition for a CBM recovers the classical CBF
feasibility condition
\[
\forall x\in B,\qquad
\exists e_x\in E_x
\quad\text{s.t.}\quad
\F h\circ f(e_x)\geq \sigma\circ h(x).
\]
where $E_x = \{ e_x \in E : p(e_x) = x\}$. 

\begin{theorem}[Safe control theorem]
\label{thm:cbm}
    Let $h \colon B \to R$ be a CBM.  Under Assumption \ref{ass:comparison}, any section $u_h \maps B \to \A_h$ of $\pi_h = p \circ a_h \colon \A_h \to B$ yields a safe controller: $k_h := a_h \circ u_h \maps B \to E$. That is, $h$ is a barrier morphism for the closed-loop system:
    \[
    f_{\rm cl} :=  f \circ k_h  
    \maps B \to \F B
    \]
    and
    $S_h = \Eq^{\leq}( \mathbf{0}_B  , h )$ is forward invariant for $\fcl$. 
\end{theorem}
\begin{proof}
The diagram for $\A_h$ given in \eqref{eq:admissible} and the section $u_h \colon B \to \A_h$ paste to give the barrier inequality \eqref{eqn:barrierdiag} for the closed-loop system: $\F h \circ f_{\rm cl} \geq \sigma \circ h$. It follows that $h$ is a barrier morphism for the autonomous closed-loop coalgebra $f_{\rm cl}$. 
Therefore, by Theorem \ref{thm:barrier}, $S_h$ is forward invariant. 
\end{proof}

\begin{remark}[Control Lyapunov morphisms]
\label{rem:clm} 
Let $V \colon B \to R$ be a morphism. Then the same construction as above with the inequality reversed gives the admissible locus: 
\[
a_V \colon \A_V=\Eq^{\le}(\F V\circ f,\sigma_V\circ V\circ p) \to E.
\]
Let $V$ be positive definite (as defined categorically in \cite{CLT1,CLT2})
and $\pi_V := p \circ a_V$ be a point surjective epi, then $V$ is a \define{control Lyapunov morphism (CLM)}\footnote{The purpose of this paper is not to develop the theory of CLM, although it is quite interesting as the ``dual'' to control barrier morphisms. }. It is straightforward to prove that any section $u_V \colon B \to \A_V$ yields a stabilizing controller $a_V \circ u_V$. Namely, for the closed loop system $\fcl := f \circ a_V \circ u_V \colon B \to \F B$ the Lyapunov inequality is satisfied: $\F V\circ\fcl\le\sigma_V\circ V$. We will leverage this reverse inequality in the context of certificates for the tracking controller. 
\end{remark}

\begin{example}[Continuous-time CBFs]
Returning to Example \ref{ex:continuous}, 
pick the comparison object $R = \R$ with the canonical order $\geq$ and equip $T R \cong \R^2$ with the order: $(r,\dot{r}) \geq (r',\dot{r}')$ iff $r = r'$ and $\dot{r} \geq \dot{r}'$. Define the comparison system $\sigma_h \colon \R \to T \R$ by $\sigma_h(r)=(r,-\alpha(r))$, for $\alpha \in \mathcal{K}_{\infty}^e$ extended class $\mathcal{K}_{\infty}$. 

For $h \colon M\to\R$, the safe set is
\[
S_h =\{x \in M  \ : \ h(x)\geq0\}
\]
and the admissible locus is given by: 
\[
\A_h
=
\left\{
(x,u) \in M \times U \ : \
D h(x)f(x,u)
\geq-\alpha(h(x))
\right\}.
\]
Point-surjectivity of $\pi_h \colon \A_h\to M$ is the usual CBF condition: 
\[
\forall x\in M,\qquad
\exists u\in U
\quad\text{s.t.}\quad
D h(x)  f(x,u)\geq -\alpha( h(x)). 
\]
A section $u_h(x) = (x,\kappa_h(x))$ is a feedback controller such that the barrier condition is satisfied: 
\[
D h(x) f(x,\kappa_h(x)) \geq - \alpha(h(x)) 
\]
and therefore $S_h$ is forward invariant for the closed-loop system: $\fcl(x) = f(x,\kappa_h(x))$. We have thus captured the classic CBF results categorically. 
\end{example}

\section{Composition of Systems and Certificates}
\label{sec:composition}

The key idea, inspired by \cite{AmesNESC2006,Tabuada2008}, is that \emph{composition can be described by a pullback}.  
We now make this idea precise in the context of control coalgebras and CBMs.

\subsection{Composition over a mediator}

Consider the case of two control coalgebras $\mathbf{P}_1$ and $\mathbf{P}_2$  (this framework can be extended to any number of systems via the limit). The essential concept is that control systems must communicate through a \define{mediator} $\mathbf{M}$, resulting in a \emph{cospan} of morphisms in $\Ctrl_{\F}(\C)$: 
\begin{equation}
 \mathbf P_1\xrightarrow{\bmap{\phi}_1}\mathbf M\xleftarrow{\bmap{\phi}_2}\mathbf P_2,
 \qquad\mathbf M=(p_0 \colon E_0\to B_0,f_0).
 \label{eq:cospan}
\end{equation}
The mediator specifies which state and input information must agree. The cospan defines constraints between states and inputs that are imposed by the interconnection between two systems. Composition enforces this agreement via the pullback.

\begin{assumption}[Existence of pullbacks]
\label{ass:pullback}
The pullbacks:
\[
\begin{tikzcd}[column sep=5mm,row sep=4mm]
    E := E_1\times_{E_0}E_2 
    \ar[r,"p_1^E"]
    \ar[d,"p_2^E"']
    \ar[dr,phantom,"\lrcorner",very near start]
    &
    E_1
    \ar[d,"\phi_1^E"]
    \\
    E_2
    \ar[r,"\phi_2^E"']
    &
    E_0
\end{tikzcd}
\quad
\begin{tikzcd}[column sep=5mm,row sep=4mm]
 B := B_1\times_{B_0}B_2 \ar[r,"p_1^B"]\ar[d,"p_2^B"']\ar[dr,phantom,"\lrcorner",very near start]&B_1\ar[d,"\phi_1^B"]\\
 B_2\ar[r,"\phi_2^B"']&B_0.
\end{tikzcd}
\]
  exist, and $\F$ preserves the state pullback: 
  \[
  \F B = \F(B_1\times_{B_0}B_2) \cong \F(B_1) \times_{\F(B_0)} \F(B_2).
  \]
\end{assumption}

For smooth manifolds, transverse pullbacks are sufficient and are preserved by $T$. Arbitrary pullbacks need not exist.

\begin{proposition}[Composition of control coalgebras]
\label{prop:physical}
Under Assumption~\ref{ass:pullback}, there is a unique control coalgebra:
\[
\mathbf P:= \mathbf P_1\times_{\mathbf M}\mathbf P_2 = (p \colon E\to B,f)
\]
given by the pullback in $\Ctrl_{\F}(\C)$: 
\[
\begin{tikzcd}[column sep=11mm,row sep=4mm]
    \mathbf P
    \ar[r, "\bmap{p}_1"]
    \ar[d, "\bmap{p}_2"']
    \ar[dr, phantom, "\lrcorner", very near start]
    &
    \mathbf P_1
    \ar[d, "\bmap{\phi}_1"]
    \\
    \mathbf P_2
    \ar[r, "\bmap{\phi}_2"']
    &
    \mathbf M.
\end{tikzcd}
\]
with projections $\bmap{p}_i = (p_i^B,p_i^E)$ satisfying:
\begin{equation}
 p_i^B\circ p=p_i\circ p_i^E,\qquad
 \F p_i^B\circ f=f_i\circ p_i^E,\quad i=1,2.
 \label{eq:induced-dynamics}
\end{equation}
\end{proposition}

\begin{remark}
\label{rmk:pullback}
The universal property of the pullback is essential to understanding properties of the composed system. It is summarized by the following system-level diagrams: 
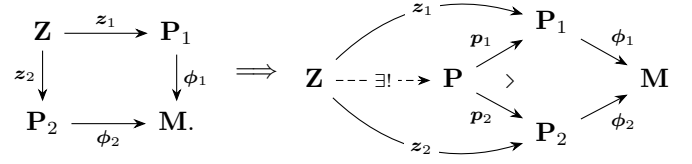
\[
\begin{tikzcd}[column sep=10mm]
    \mathbf Z 
    \ar[r,"\bmap z_1"]
    \ar[d,"\bmap z_2"']
    &
    \mathbf P_1
    \ar[d,"\bmap \phi_1"]
    \\
    \mathbf P_2
    \ar[r,"\bmap \phi_2"']
    &
    \mathbf M.
\end{tikzcd}
\ \implies 
\begin{tikzcd}[column sep = scriptsize, row sep = 2mm]
    &&&
    \mathbf P_1
    \arrow[dr, "\bmap \phi_1"]
    \\
    \mathbf Z
    \arrow[urrr, bend left, "\bmap z_1"description]
    \arrow[rr, "\exists!"description, dashed]
    \arrow[drrr, bend right, "\bmap  z_2"description]
    &&
    \mathbf P
    \arrow[ur, "\bmap p_1"]
    \arrow[dr, "\bmap p_2"']
    \arrow[rr, phantom, "\lrcorner"{rotate = 45}, pos = 0.2]
    &&
    \mathbf M
    \\&&&
    \mathbf P_2
    \arrow[ur, "\bmap \phi_2"']
\end{tikzcd}
\]
Hence, for a system $\mathbf Z$ with compatible maps to the subsystems $\mathbf P_1$ and $\mathbf P_2$, there is a unique induced map to $\mathbf P$. This map will be essential in extending certificates (e.g., barrier morphisms) from subsystems to the composed closed-loop system. 
\end{remark}

\begin{proof}
The following diagram commutes.
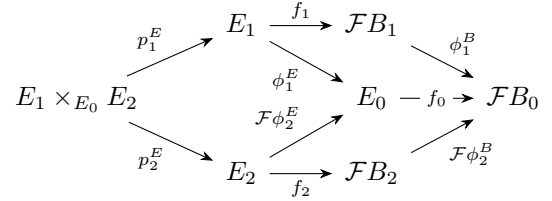
\[
\begin{tikzcd}[row sep = 4mm]
    &
    E_1
    \arrow[r, "f_1"]
    \arrow[dr, "\phi_1^E"']
    &
    \F B_1
    \arrow[dr, "\phi_1^B"]
    \\
    E_1 \times_{E_0} E_2
    \arrow[ur, "p_1^E"]
    \arrow[dr, "p_2^E"']
    &&
    E_0
    \arrow[r, "f_0"description]
    &
    \F B_0
    \\&
    E_2
    \arrow[ur, "\F\phi_2^E"]
    \arrow[r, "f_2"']
    &
    \F B_2
    \arrow[ur, "\F\phi_2^B"']
\end{tikzcd}
\]
So the maps $f_i \circ p_i^E \colon E \to \F B_i$, $i=1,2$, form a cone over $\mathcal F B_1\to\mathcal F B_0\leftarrow\mathcal F B_2$. Since $\mathcal F$ preserves the state pullback, universal property of pullback gives a unique $f \colon E\to\mathcal F B$ such that the following diagram commutes: 
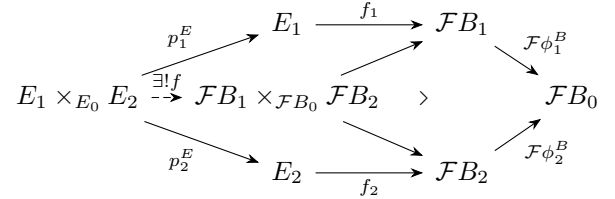
\[
\begin{tikzcd}[column sep = small, row sep = 4mm]
    &
    E_1
    \arrow[r, "f_1"]
    &
    \F B_1
    \arrow[dr, "\F\phi_1^B"]
    \\
    E_1 \times_{E_0} E_2
    \arrow[ur, "p_1^E"]
    \arrow[r, "\exists!f", dashed]
    \arrow[dr, "p_2^E"']
    &
    \F B_1 \times_{\F B_0} \F B_2
    \arrow[ur, ""]
    \arrow[rr, phantom, "\lrcorner"{rotate = 45}, pos = 0.2]
    \arrow[dr, ""']
    &&
    \F B_0
    \\&
    E_2
    \arrow[r, "f_2"']
    &
    \F B_2
    \arrow[ur, "\F\phi_2^B"']
\end{tikzcd}
\]
Setting $p := \pi_B \circ f$ makes $(p,f)$ into a control coalgebra automatically. By standard arguments, $\mathbf P$ satisfies the universal property of the desired pullback in $\Ctrl_{\F}(\C)$. 
\end{proof}

\begin{example}[Continuous-time interconnection]
\label{ex:interconnection}
Consider again continuous-time systems cf. Example \ref{ex:continuous}. 
Namely, for $\F=T$, let $\mathbf P_i$ have $E_i=M_i\times U_i$ and dynamics in coordinates $\dot x_i=f_i(x_i,u_i)$. A \define{stateless port mediator} is
\[
 \mathbf M_W=(p_0 := !_W \colon W\to1, f_0 := !_{W}),\qquad  T1\cong1 = \{*\}
\]
where $f_0(w) = !_{W}(w) \equiv *$. 
Any map $\theta_i \colon E_i\to W$ defines a morphism of control systems $\bmap{\theta}_i := (!_{M_i},\theta_i) \colon \mathbf P_i\to\mathbf M_W$. Consequently,
\begin{equation}
    B=M_1\times M_2,\quad E=(M_1\times U_1)\times_W(M_2\times U_2).
\label{eq:ports}
\end{equation}

\vspace{0.2cm}
\noindent \underline{Cascade interconnection:}
Take $W=U_1$ wherein $p_0 \colon U_1 \to 1$, $\theta_1(x_1,u_1)=u_1$, and $\theta_2(x_2,u_2)=y_2(x_2)$ for any $y_2 \colon M_2 \to U_1$. Then $u_1=y_2(x_2)$, so $E\cong M_1\times M_2\times U_2$ and: 
\begin{equation}
 \dot x_1=f_1(x_1,y_2(x_2)),\qquad \dot x_2=f_2(x_2,u_2).
 \label{eq:cascade}
\end{equation}
The remaining external input is $u_2$, resulting in the cascade interconnection: 
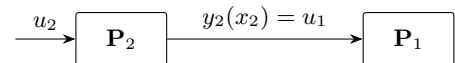
\[
\begin{tikzpicture}[>=Stealth,sys/.style={draw,rounded corners=1pt,minimum width=12mm,minimum height=7mm},font=\small]
 \node[sys] (p2) at (0,0) {$\mathbf P_2$};
 \node[sys] (p1) at (3.8,0) {$\mathbf P_1$};
 \draw[->] (-1.4,0)--node[above] {$u_2$}(p2);
 \draw[->] (p2)--node[above] {$y_2(x_2)=u_1$}(p1);
\end{tikzpicture}
\]

\vspace{0.2cm}
\noindent \underline{Feedback interconnection:}
Let $W=U_1\times U_2$, $p_0(u_1,u_2) \equiv * $, and pick outputs $y_1 \colon M_1\to U_2$, $y_2 \colon M_2\to U_1$. With
\[
 \theta_1(x_1,u_1)=(u_1,y_1(x_1)),\quad
 \theta_2(x_2,u_2)=(y_2(x_2),u_2),
\]
the pullback enforces $u_1=y_2(x_2)$ and $u_2=y_1(x_1)$. Hence $E \cong M_1 \times M_2$ and the composition is autonomous:
\begin{equation}
 \dot x_1=f_1(x_1,y_2(x_2)),\qquad
 \dot x_2=f_2(x_2,y_1(x_1)).
 \label{eq:feedback}
\end{equation}
This yields the familiar feedback interconnection diagram:
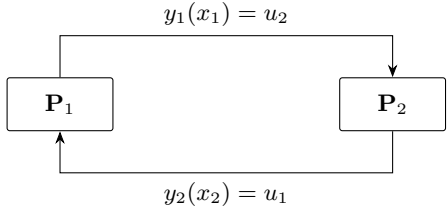
\[
\begin{tikzpicture}[>=Stealth,sys/.style={draw,rounded corners=1pt,minimum width=14mm,minimum height=7mm},font=\small]
 \node[sys] (p1) at (0,0) {$\mathbf P_1$};
 \node[sys] (p2) at (4.4,0) {$\mathbf P_2$};
 \draw[->] (p1.north)--++(0,.55)-|node[pos=.25,above] {$y_1(x_1)=u_2$}(p2.north);
 \draw[->] (p2.south)--++(0,-.55)-|node[pos=.25,below] {$y_2(x_2)=u_1$}(p1.south);
\end{tikzpicture}
\]
\end{example}

\subsection{Comparison systems and compatible certificates}

To compose guarantees for systems framed as certificates, e.g., control barrier morphisms, it is necessary to compose the target comparison systems for the layered system architecture: $\mathbf{P}_1 \rightarrow \mathbf{M} \leftarrow \mathbf{P}_2$. 
Choose autonomous targets $\mathbf R_i=(R_i,\sigma_i)$, $i = 0$ (for $\mathbf{M}$) and $i = 1, 2$ (for $\mathbf{P}_i$), and connecting maps $\bmap n_i = (n_i,n_i)$: 
\begin{equation}
    \mathbf R_1\xrightarrow{\bmap n_1}\mathbf R_0\xleftarrow{\bmap n_2}\mathbf R_2,
    \qquad \F n_i\circ\sigma_i=\sigma_0\circ n_i.
\end{equation}
Assume $R=R_1\times_{R_0}R_2$ exists and is preserved by $\F$. Proposition~\ref{prop:physical} gives the autonomous pullback: 
\begin{eqnarray}
\label{eqn:autonmoustargetcomp}
\begin{tikzcd}[column sep=11mm,row sep=6mm]
    \mathbf \QT= (R,\sigma) := \mathbf{R}_1 \times_{\mathbf{R}_0} \mathbf{R}_2 
    \ar[r,"\bmap p_1^R"]
    \ar[d,"\bmap p_2^R"']
    \ar[dr,phantom,"\lrcorner",very near start]
    &
    \mathbf R_1
    \ar[d,"\bmap n_1"]\\
    \mathbf R_2
    \ar[r,"\bmap n_2"']
    &
    \mathbf R_0
\end{tikzcd}
\end{eqnarray}
with projections $\bmap p_i^R = (p_i^R,p_i^R)$ where $p_i^R \colon R \to R_i$. Give each $\F R_i$, $i=1,2$, a preorder $\ge_i$ and equip $\F R$ with the \define{projection-induced preorder}:
\begin{equation}
    \xi\ge\eta\quad\Longleftrightarrow\quad
    \F p_i^R\circ\xi\ge_i\F p_i^R\circ\eta,
    \quad i=1,2,
    \label{eq:induced-order}
\end{equation}
for every pair $\xi,\eta \colon X\to\F R$. Precomposition preserves this preorder. The two component orders need not compare the same quantities or use the same direction of scalar inequality. 
For example, if we are considering a control barrier function $h \colon M_1 \to R_1 = \R$ and a control Lyapunov function $V \colon M_2 \to R_2 = \R_{\geq 0}$, the corresponding preorders on $T\R$ and $T\R_{\geq 0}$ would be $\ge_1 \,=\, \ge$ and $\ge_2 \,=\, \le$ (since Lyapunov and barrier functions have flipped orders). If the target mediator is terminal, $\mathbf R_0 = 1$, this induces a mixed componentwise order on the composed target.

\subsection{Main Result} 

We can now present the main result on the composition of certificates, i.e., control barrier morphisms. For this theorem, under Assumption \ref{ass:pullback}, let
\begin{eqnarray}
\label{eqn:PQassumption}
 \begin{aligned}
 \mathbf P&=(p \colon E\to B,f)
   =\mathbf P_1\times_{\mathbf M}\mathbf P_2,\\
 \mathbf Q&=(R,\sigma)
   =\mathbf R_1\times_{\mathbf R_0}\mathbf R_2
 \end{aligned}
\end{eqnarray}
\newpage
\noindent be the composed control system from Proposition \ref{prop:physical} and the autonomous target system \eqref{eqn:autonmoustargetcomp}, respectively, where $R$ is posetal and $\F R$ has the projection-induced preorder \eqref{eq:induced-order}.

\begin{theorem}[Composition of control barrier morphisms]
\label{thm:composition}
Let $\mathbf{P}$ be the composed control system and $\mathbf{Q}$ be the autonomous target as in \eqref{eqn:PQassumption}. Suppose there are CBMs $\chi_i \colon B\to R_i$, $i = 1, 2$, for $\mathbf{P}$ with admissible loci $\A_{\chi_i}$ such that the diagram below commutes: 
\begin{equation}
\label{eq:compatible}
\begin{tikzcd}[row sep=4mm]
    B
    \arrow[r, "\chi_1"]
    \arrow[d, "\chi_2"']
    &
    R_1
    \arrow[d, "n_1"]
    \\
    R_2
    \arrow[r, "n_2"']
    &
    R_0
\end{tikzcd} 
\quad \rightsquigarrow \quad 
 n_1\circ\chi_1=n_2\circ\chi_2,
\end{equation}
and the following pullback 
exists, inducing a map $a \colon \mathcal A\to E$:
\[
\begin{tikzcd}[column sep=8mm, row sep=4mm]
    \A := \A_{\chi_1}\times_E \A_{\chi_2}
    \ar[r,"p_1^\A"]
    \ar[d,"p_2^\A"']
    \ar[dr,phantom,"\lrcorner",very near start]
    &
    \A_{\chi_1} \ar[d,hook,"a_{\chi_1}"]
    \\
    \A_{\chi_2} \ar[r,hook,"a_{\chi_2}"']
    &
    E
\end{tikzcd}
\quad \rightsquigarrow \quad 
\begin{aligned}
    a:= & a_{\chi_1}\circ p_1^\A \\
    = & a_{\chi_2}\circ p_2^\A
\end{aligned}
\]  
If $\pi_A := p \circ a \colon \A \to B$ is an epi and point-surjective,
there exists a unique CBM $\chi \colon B\to R$ with
$p_i^R\circ\chi=\chi_i$, $i=1,2$, satisfying
\begin{equation}
\begin{tikzcd}[cramped,column sep=4.6mm, row sep = 2mm]
	& {E } & B & R \\
	{\A_\chi\cong\mathcal A} \\
	& E & \F B & \F R
	\arrow["p", from=1-2, to=1-3]
	\arrow["\chi", from=1-3, to=1-4]
	\arrow["\sigma", from=1-4, to=3-4]
	\arrow["{a}", from=2-1, to=1-2]
	\arrow["{a}"', from=2-1, to=3-2]
	\arrow[from=3-2, to=1-4, phantom, sloped, "\geq", pos = 0.5]
	\arrow["f~"', from=3-2, to=3-3]
	\arrow["\F \chi~"', from=3-3, to=3-4]
    \arrow["\pi_{\A}", bend left=60, from=2-1, to=1-3]
\end{tikzcd}
\rightsquigarrow  \F\chi\circ f\circ a
 \ge \sigma\circ\chi\circ\pi_{\mathcal A}.
 \label{eq:eq:composed-admissible-cbm}
\end{equation}
\end{theorem}

\begin{remark}
This theorem can be visualized by the diagram: 
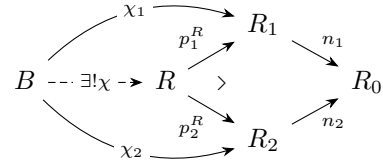
\[
\begin{tikzcd}[column sep = scriptsize, row sep = 2mm]
    &&&
    R_1
    \arrow[dr, "n_1"]
    \\
    B
    \arrow[urrr, bend left, "\chi_1"description]
    \arrow[rr, "\exists!\chi"description, dashed]
    \arrow[drrr, bend right, "\chi_2"description]
    &&
    R
    \arrow[ur, "p^R_1"]
    \arrow[dr, "p^R_2"']
    \arrow[rr, phantom, "\lrcorner"{rotate = 45}, pos = 0.2]
    &&
    R_0
    \\&&&
    R_2
    \arrow[ur, "n_2"']
\end{tikzcd}
\]
where the $\chi_i$ are state maps that lax commute on the dynamics according to \eqref{eq:compatible}, and the remaining arrows strictly commute. 
Therefore, certificates (CBMs) from $\mathbf{P}$ to each of the targets $\mathbf{R}_i$ induce a unique certificate to the composed target. 
Additionally, if there are component CBMs $\bar\chi_i \colon B_i\to R_i$ associated with each of the subsystems $\mathbf{P}_i$, Theorem \ref{thm:composition} can be applied with $\chi_i = \bar \chi_i \circ p_i^B$ assuming the conditions of the theorem hold. 
The existence of sections of $\pi_{\chi_1}$ and $\pi_{\chi_2}$ does not imply the existence of a section of $\pi_{\A}$.
\end{remark}

\begin{proof}
The compatibility of $\chi_i$ in \eqref{eq:compatible} gives a unique map $\chi \colon B\to R$ with
$p_i^R\circ\chi=\chi_i$.
For any $q \colon X\to E$, strictness of the target projections
and the induced preorder give:
\[
 \begin{aligned}
 &\F\chi\circ f\circ q
       \ge\sigma\circ\chi\circ p\circ q\\
 &\quad\Longleftrightarrow\quad
 \F\chi_i\circ f\circ q
       \ge_i\sigma_i\circ\chi_i\circ p\circ q,
       \quad i=1,2. \\
 &\quad\Longleftrightarrow\quad q \textrm{ factors through both } a_{\chi_1},a_{\chi_2},\\
 &\quad\Longleftrightarrow\quad q \textrm{ factors through } a. 
 \end{aligned}
\]
\newpage
Here, the first equivalence follows from the induced preorder \eqref{eq:induced-order}, the identities $p_i^R \circ \chi = \chi_i$, and strictness of the target projections $\bmap p_i^R$. The remaining equivalences follow from the universal properties of the admissible loci and their pullback. Therefore, $\A$ satisfies the universal property of the admissible locus of $\chi$, and thus $\A_{\chi} \cong \A$ over $E$ and $\pi_{\chi}$ can be identified with $\pi_{\A}$. Taking $q = a$ yields \eqref{eq:eq:composed-admissible-cbm} and $\chi$ is a control barrier morphism. 
\end{proof}

\begin{corollary}[Safety by certificate composition]
\label{cor:composition}
    Suppose that there is a composed target system $\QT = (R,\sigma)$, where $R$ is posetal with distinguished zero $0_R \colon 1 \to R$ and $\F R$ has the induced order~\eqref{eq:induced-order}, such that Assumption~\ref{ass:comparison} is satisfied. Under Theorem~\ref{thm:composition}, every section $u \colon B \to \A$ renders
    \begin{equation}
     S_\chi=\Eq^{\le}(\zero B,\chi)\hookrightarrow B
     \label{eq:composed-safe}
    \end{equation}
    forward invariant for $\fcl=f\circ a \circ u \colon B \to \F B$.
\end{corollary}

Note that if $R_+ := \Eq^{\le}(\zero R, \id_R)$ exists, the universal property of the lax equalizer and the pullback imply that:
\begin{equation}
    S_\chi \cong B \times_R R_+ .
    \label{eq:safe-pullback}
\end{equation}
Indeed, $q \colon X \to B$ satisfies $\zero X \le \chi \circ q$ exactly when $\chi \circ q$ factors through $R_+$. Thus the standard safe subobject is the inverse image of target positivity: $S_\chi = \chi^{-1}(R_+)$.

\begin{remark}[General composition]
    Composition as the pullback extends to any number of systems. In this case, a graph $\Gamma$ describes the interconnection structure, with a system $\mathbf{P}_v$ at each node $v$ and a mediator $\mathbf{M}_e$ at each edge $e$, giving the cospan
    \[
    \mathbf{P}_{\mathrm{source}(e)} \rightarrow \mathbf{M}_e \leftarrow \mathbf{P}_{\mathrm{target}(e)},
    \]
    and the composed system is the limit of the resulting diagram, provided this limit exists and $\F$ preserves the corresponding limit of state spaces (cf.~\cite{AmesNESC2006} for tagged systems). For example, the classic three-layer control architecture, consisting of a planning layer, a trajectory generation layer, and a real-time control layer, composes into a single control coalgebra
    \[
     \mathbf{P} := \lim\left( \mathbf P_1\longrightarrow\mathbf M_{12}\longleftarrow
     \mathbf P_2\longrightarrow\mathbf M_{23}\longleftarrow\mathbf P_3 \right).
    \]
    We expect Theorem~\ref{thm:composition} to extend to this case: given CBMs $\chi_v \colon B \to R_v$ for $\mathbf{P}$, one for each node, that are compatible over a diagram of target systems of the same shape, the universal property of the limit induces a CBM on the composed system, provided the joint admissible locus projects onto $B$ by a point-surjective epi.
\end{remark}

\section{Layered Safety for Reduced-Order Models}
\label{sec:application}

We return to the motivation given in Section \ref{sec:background}: layered safety architectures leveraging RoMs to guarantee safety on FoMs (see Fig. \ref{fig:architecture} for an illustration of this architecture). The goal of this section is to apply the categorical results on composition to this specific case. We show that the result in \eqref{eq:classical-result} follows systematically as a result of the categorical framework. Throughout this section we work in $\C = \Man$. 

We begin by noting that assumptions on the interface between the RoM and FoM can be understood categorically. 

\newpage

\begin{assumption}[Projection compatibility]
\label{ass:projection}
There are smooth maps $\Pi \colon X\to Z$ and $\Psi \colon X\to U_{\nu}$ such that, for every $(x,v)\in X\times U_v$, it follows that: 
\begin{equation}
    D\Pi(x)f_{\rm F}(x,v)=f_{\rm R}(\Pi(x),\Psi(x)).
\label{eq:projection}
\end{equation}
Equivalently, the following dynamics square commutes:
\begin{equation}
\begin{tikzcd}[column sep=10mm,row sep=6mm]
    X\times U_v \ar[r,"{(\Pi,\Psi)}"]\ar[d,"f_{\rm F}"']&Z\times U_{\nu} \ar[d,"f_{\rm R}"]\\
    T X\ar[r,"T\Pi"']&TZ.
\end{tikzcd}
\label{eq:projection-diagram}
\end{equation}
where $(\Pi,\Psi)(x,v)=(\Pi(x),\Psi(x))$.
\end{assumption}

Since the FoM is obtained from a tracking controller $u=K(x,v)$ applied to the full order dynamics: $f_{\rm F}(x,v)=F(x)+G(x)K(x,v)$, sufficient conditions \cite{Compton2025} for \eqref{eq:projection-diagram} to hold are: $D\Pi(x)\,G(x) =0$ and $D\Pi(x)\,F(x)=f_{\rm R}(\Pi(x),\Psi(x))$. 

\subsection{The safety-filter mediator and the closed-loop pullback}

We can mirror the conceptual safety architecture in Figure \ref{fig:architecture} with a precise categorical representation as a cospan: 
\begin{equation}
\label{eq:physical-interfaces}
\begin{tikzpicture}[
  >=latex,
  system/.style={
    draw=black!65,
    rounded corners=2pt,
    fill=white,
    text width=.63\linewidth,
    align=center,
    inner xsep=1.2mm,
    inner ysep=.9mm,
    font=\footnotesize
  },
  morphism/.style={
    ->,
    line width=.65pt
  },
  maplabel/.style={
    fill=white,
    inner xsep=1.0mm,
    inner ysep=.4mm,
    align=left,
    font=\scriptsize
  }
]

\node[
  system,
  minimum height=10mm
] (rom) {
  $\displaystyle
  \mathbf P_{\rm RoM}
  =
  \bigl(p_Z \colon Z\times U_\nu\to Z,f_{\rm R}\bigr)
  $
  \\ [.2mm]
  $\displaystyle
  \dot z=f_{\rm R}(z,\nu)
  $
};

\node[
  system,
  below=10mm of rom,
  minimum height=10mm
] (med) {
  $\displaystyle
  \mathbf M_{\rm sf}
  =
  \bigl(
    p_0 \colon Z\times U_\nu\times U_v\to Z,f_{\rm sf}
  \bigr)
  $
  \\ [.2mm]
  $\displaystyle
  \dot z=f_{\rm sf}(z,\nu,v)=f_{\rm R}(z,\nu)
  $
};

\node[
  system,
  below=10mm of med,
  minimum height=10mm
] (fom) {
  $\displaystyle
  \mathbf P_{\rm FoM}
  =
  \bigl(p_X \colon X\times U_v\to X,f_{\rm F}\bigr)
  $
  \\ [.2mm]
  $\displaystyle
  \dot x=f_{\rm F}(x,v)
  $
};

\draw[morphism]
  (rom.south) --
  node[
    maplabel,
    midway,
    right=1.5mm
  ] {
    $\begin{aligned}
        \phi_R^B(z) & =z \\
        \phi_R^E(z,\nu)
       & =(z,\nu,\kappa_{\rm safe}(z))
    \end{aligned}$
  }
  (med.north);

\draw[morphism]
  (fom.north) --
  node[
    maplabel,
    midway,
    right=1.5mm
  ] {
    $\begin{aligned}
        \phi_F^B(x) & =\Pi(x) \\
        \phi_F^E(x,v)
       & =(\Pi(x),\Psi(x),v)
    \end{aligned}
    $
  }
  (med.south);

\end{tikzpicture}
\end{equation}
Here, the mediator has copies of $U_{\nu}$ and $U_v$ to receive control interfaces from the RoM and FoM. It is these interfaces that encode the safety filter itself, $\kappa_{\rm safe}$. 
Throughout this section, let $\kappa_{\rm safe}$ be a smooth safety filter satisfying the constraint in \eqref{eq:qp}, not necessarily its exact optimizer, which exist by \cite{cohen2023characterizing}. 

The closed-loop system is obtained via the pullback. 

\begin{proposition}[Composition is the safety filtered FoM]
\label{prop:closedloop}
Under Assumption~\ref{ass:projection}, the cospan~\eqref{eq:physical-interfaces} has the pullback
\begin{equation}
 \PC:=\PR\times_{\MS}\PF\cong(\id_X \colon X\to X,\fcl),
 \label{eq:physical-pullback}
\end{equation}
where $\fcl$ is given by~\eqref{eq:closedloop}.
\end{proposition}

\begin{proof}
The pullback imposes equality in the mediator: 
\[
 (z,\nu,\ks(z))=(\Pi(x),\Psi(x),v).
\]
This determines $z=\Pi(x)$, $\nu=\Psi(x)$, and $v=\ks(\Pi(x))$ uniquely. 
Hence the state space pullback and control bundle pullback are both $X$, with the identity bundle projection. 
The FoM projection determines the dynamics of the pullback as $f_{\rm F}(x,\ks(\Pi(x)))$. Assumption~\ref{ass:projection} verifies the RoM projection, completing the pullback cone. 
\end{proof}

As in Section \ref{sec:background}, and specifically \eqref{eq:classical-bounds}, we consider a smooth non-negative \define{tracking function} $V \colon X\to\R_{\geq 0}$ satisfying: 

\newpage

\begin{assumption}[Tracking]
\label{ass:tracking}
On $X$, for constants $\gamma,\rho>0$ and $\lambda>\alpha\ge0$, assume: 
\begin{equation}
 \begin{aligned}
 V(x)&\ge \rho\|\Psi(x)-\ks(\Pi(x))\|^2,\\
 DV(x)\fcl(x)&\le-\lambda V(x).
 \end{aligned}
 \label{eq:tracking}
\end{equation}
\end{assumption}

Let $b:=h\circ\Pi$. To recover the \emph{barrier dynamics} in \eqref{eq:classical-bounds}, by Assumption~\ref{ass:projection}, the RoM CBF inequality, and Assumption~\ref{ass:tracking}: 
\begin{align}
    \dot b
    &=Dh(\Pi(x))f_{\rm R}(\Pi(x),\Psi(x)) \nonumber\\
    &=L_fh(\Pi(x))+L_gh(\Pi(x))\ks(\Pi(x))
       +L_gh(\Pi(x))e(x) \nonumber\\
    & \ge-\alpha b + \tfrac{\| L_g h(z) \|^2}{4 \gamma \rho} + L_g h(z) e    \nonumber\\
    & \ge-\alpha b-\gamma\rho\|e\|^2 \ge-\alpha b-\gamma V.
    \label{eq:two-bounds}
\end{align}
where here $z = \Pi(x)$ and the last inequalities follow from completing the square and \eqref{eq:tracking}. 
The end result is two certificates on the closed loop: $b$, built from the RoM barrier $h$, satisfying $\dot b \geq -\alpha b - \gamma V$, and the tracking function $V$, satisfying $\dot V \leq -\lambda V$. These two certificates will form the basis for constructing the target systems, and will ultimately generate the induced certificate for $\PC$.

\subsection{Comparison targets and their pullback}

The barrier lower bound~\eqref{eq:two-bounds} and the tracking decay
assumption~\eqref{eq:tracking} are encoded by the
comparison systems $\mathbf{R}_h$ and $\mathbf{R}_V$ via the cospan: 
\begin{equation}
\label{eq:comparison-cospan}
\begin{tikzpicture}[
  >=latex,
  baseline={([yshift=-.5ex]rzero.center)},
  system/.style={
    draw=black!65,
    rounded corners=2pt,
    fill=white,
    text width=.54\linewidth,
    align=center,
    inner xsep=.8mm,
    inner ysep=.6mm,
    outer sep=0pt,
    font=\footnotesize
  },
  morphism/.style={->,line width=.65pt},
  maplabel/.style={font=\scriptsize,inner sep=0pt,outer sep=0pt}
]

\node[system] (rh) {
  $\mathbf R_h
    = \bigl(\R^2_{(b,r)},\,\sigma_h\bigr)$
  \\[.25mm]
  $\sigma_h ~ \textrm{dynamics:} ~  \begin{cases}\dot b&=-\alpha b-\gamma r\\\dot r&=-\lambda r\end{cases}$
};

\node[system,below=6mm of rh] (rzero) {
  $\mathbf R_0
    = \bigl(\R_r,\,\sigma_0\bigr)$
  \\[.25mm]
  $\sigma_0 ~ \textrm{dynamics:} ~ \dot r=-\lambda r$
};

\node[system,below=6mm of rzero] (rv) {
  $\mathbf R_V
    = \bigl(\R_s,\,\sigma_V\bigr)$
  \\[.25mm]
  $\sigma_V ~ \textrm{dynamics:} ~ \dot s=-\lambda s$
};

\draw[morphism]
  (rh.south) --
  node[midway,right=1.4mm,maplabel] {$n_1=\mathrm{pr}_r$}
  (rzero.north);

\draw[morphism]
  (rv.north) --
  node[midway,right=1.4mm,maplabel] {$n_2=\mathrm{id}_{\mathbb R}$}
  (rzero.south);

\end{tikzpicture}
\end{equation}
At equal base points, the preorders on $\mathbf{R}_h$ and $\mathbf{R}_V$ are:
\begin{equation}
\begin{aligned}
    \textrm{Preorder on } T \R^2_{b,r} \hspace{-1mm}: &   \\
    \textrm{Preorder on } T \R_s \hspace{-1mm}: & 
\end{aligned} 
\quad 
 \begin{aligned}
 (\dot b,\dot r)\ge_h(\dot b',\dot r')&\iff\dot b\ge\dot b',\\
 \dot s\ge_V\dot s'&\iff\dot s\le\dot s'.
 \end{aligned}
 \label{eq:mixed}
\end{equation}
The first order tests the barrier lower bound, while the second order tests the tracking upper bound.

The goal is to find a  composed comparison system $\mathbf{Q} \cong \mathbf{R}_h \times_{\mathbf{R}_0} \mathbf{R}_V$ such that the conditions of Corollary \ref{cor:composition} are satisfied.
Set $a := \lambda - \alpha > 0$. The composed comparison system $\mathbf{Q}$ is given in \((b,\ell)\)-coordinates with corresponding projections $\bmap{p}_h^R$ and $\bmap{p}_V^R$ as given in the following pullback diagram (see Remark \ref{rem:Qjustification} for additional details on choosing $\mathbf{Q}$):
\begin{equation}
\label{eq:comparison-pullback}
\begin{tikzpicture}[
  >=latex,
  baseline={([yshift=-.5ex]rzero.center)},
  system/.style={
    draw=black!65,
    rounded corners=2pt,
    fill=white,
    text width=.41\linewidth,
    align=center,
    inner xsep=.8mm,
    inner ysep=.6mm,
    outer sep=0pt,
    font=\footnotesize
  },
  morphism/.style={->,line width=.65pt},
  maplabel/.style={font=\scriptsize,inner sep=.4pt,outer sep=0pt,align=center}
]

\node[system] (q) {
  $\mathbf Q
    = \bigl(\R^2_{(b,\ell)},\,\sigma\bigr)$
  \\[.25mm]
  $\sigma ~ \textrm{dynamics:} ~ \begin{cases}\dot b&=-\lambda b+\ell\\\dot\ell&=-\alpha\ell\end{cases}$
};

\node[system,right=31mm of q,text width=.14\linewidth] (rh) {
  $\mathbf R_h$
};

\node[system,below=12mm of q,text width=.14\linewidth] (rv) {
  $\mathbf R_V$
};

\node[system,right=42.8mm of rv,text width=.14\linewidth] (rzero) {
  $\mathbf R_0$
};

\draw[morphism]
  (q.east) --
  node[midway,above=1.2mm,maplabel]
  {$\displaystyle
    p_h^R(b,\ell)
    =
    \left(
      b,\,
      \frac{ab-\ell}{\gamma}
    \right)$}
  (rh.west);

\draw[morphism]
  (q.south) --
  node[midway,left=1.4mm,maplabel]
  {$\displaystyle
    p_V^R(b,\ell)
    =
    \frac{ab-\ell}{\gamma}$}
  (rv.north);

\draw[morphism]
  (rh.south) --
  node[midway,left=1.4mm,maplabel] {$n_1=\mathrm{pr}_r$}
  (rzero.north);

\draw[morphism]
  (rv.east) --
  node[midway,below=1.2mm,maplabel] {$n_2=\mathrm{id}_{\mathbb R}$}
  (rzero.west);

\node at ($(q.south east)!0.2!(rzero.north west)$) {$\lrcorner$};

\end{tikzpicture}
\end{equation}

\newpage
The fact that $\mathbf{Q}$ is the pullback of \eqref{eq:comparison-cospan} follows from the observation that $L \colon \R^2_{(b,r)} \to \R^2_{(b,\ell)}$, given by $L(b,r)=(b,ab-\gamma r)$ is invertible and satisfies $TL\circ \sigma_h=\sigma\circ L$. 

The dynamics preorder on \(\mathbf Q\), i.e., the preorder on $T \R^2_{(b,\ell)}$, is the one induced by the
projections in \eqref{eq:comparison-pullback} per \eqref{eq:induced-order}. In \((b,\ell)\)-coordinates at a fixed base point this is given by: 
\begin{equation}
(\dot{b},\dot{\ell}) \geq_{\mathbf{Q}} (\dot{b}',\dot{\ell}') \quad \iff \quad 
\dot{b} \geq \dot{b}', \quad 
\dot{\ell} - \dot{\ell}' \geq a (\dot{b} - \dot{b}') . \nonumber
\end{equation}
Since $a > 0$, this implies the ordinary componentwise preorder on $T \R^2_{(b,\ell)}$. Finally, the dynamics on $\mathbf{Q}$ 
imply that $\R^2_{\geq 0}$ is forward invariant verifying Assumption \ref{ass:comparison}.

\begin{remark}[Finding the comparison system]
\label{rem:Qjustification}
It is important to note that the choice of $\mathbf{Q}$ was guided by the need to satisfy Assumption \ref{ass:comparison} for the induced preorder \eqref{eq:induced-order} as required by Corollary \ref{cor:composition}. For example, a naive (but valid) choice for the pullback is $\mathbf{R}_h = (\R^2_{(b,r)},\sigma_h) \cong \mathbf{R}_h \times_{\mathbf{R}_0} \mathbf{R}_V$. Yet the dynamics $\sigma_h$ for this system do not satisfy Assumption \ref{ass:comparison}: at $b=0$, $r>0$, it follows that $\dot b=-\gamma r<0$. 
Now, to understand how we arrived at $\mathbf{Q}$, we can determine the largest forward invariant subset of $\{ (b,r) : b \geq 0\}$: 
\[
 K :=\{(b,r):b\ge0,\;ab-\gamma r\ge0\} .
\]
Defining $L(b,r)=(b,ab-\gamma r)$ as above, it follows that $K = \{ (b,r) : L(b,r) \geq 0\}$ and $L(K)=\mathbb R_{\ge0}^2$.
Transporting the pullback through $L$ yields
$\mathbf Q$ as defined above. 
\end{remark} 

\subsection{The induced certificate and full-order model safety}

The goal is now to apply the main result, Theorem \ref{thm:composition}, to obtain a control barrier function for the closed-loop system $\mathbf{P}_{\rm cl}$ representing the layered architecture. To this end, consider the barrier certificate $\chi_h :=(b,V) \colon X\to\R^2_{(b,r)}$ and the tracking certificate $\chi_V := V\colon X\to\R_s$. These are compatible since $\pr_r\circ\chi_h=\chi_V$. Moreover, \eqref{eq:two-bounds} and \eqref{eq:tracking} imply that: 
\[
T\chi_h\circ\fcl\ge_h\sigma_h\circ\chi_h,\qquad
 T\chi_V\circ\fcl\ge_V\sigma_V\circ\chi_V
\]
for the chosen target systems (by design).

The target pullback $\mathbf{Q}$ induces a unique morphism: 
\[
\begin{tikzcd}[column sep = scriptsize, row sep = 2mm]
    &&&
    \RH
    \arrow[dr, "\bmap n_1"]
    \\
    \PC
    \arrow[urrr, bend left, "{\chi_h = (b,V)}"description]
    \arrow[rr, "\exists!\chi"description, dashed]
    \arrow[drrr, bend right, "\chi_V=V"description]
    &&
    \QT
    \arrow[ur, "\bmap p_h^R"]
    \arrow[dr, " \bmap p_V^R"']
    \arrow[rr, phantom, "\lrcorner"{rotate = 45}, pos = 0.2]
    &&
    \RO
    \\&&&
    \RV
    \arrow[ur, "\bmap n_2"']
\end{tikzcd}
\]
where the $\chi$ arrows denote state maps satisfying lax dynamics inequalities \eqref{eq:tracking}, \eqref{eq:two-bounds}. 
Applying Theorem \ref{thm:composition} with $p_h^R$ and $p_V^R$ implies that the universal property uniquely determines: 
\begin{equation}
 \chi(x)=(b(x),H(x)),\qquad H(x)=a\,b(x)-\gamma V(x). 
 \label{eq:induced-H}
\end{equation}
Moreover, the conclusion that $T\chi\circ\fcl\ge\sigma\circ\chi$ yields: 
\begin{equation}
 \dot b\ge-\lambda b+H,\qquad \dot H\ge-\alpha H.
 \label{eq:H-comparison}
\end{equation}
This allows us to establish through categorical means safety of $\PC$, and specifically the closed-loop FoM first considered in \eqref{eq:closedloop}.  Since $\PC$ is autonomous, $\A_{\chi_h} \cong \A_{\chi_V} \cong X$. 

\newpage

\begin{theorem}[Layered safety]
\label{thm:layered}
Suppose Assumptions~\ref{ass:projection} and~\ref{ass:tracking} hold, and that there exists a smooth safe controller $\nu = \kappa_{\rm safe}(z)$ for the RoM $\dot{z} = f_{\rm R}(z,\nu)$ satisfying: 
\[
D h(z) f_{\rm R}(z,\kappa_{\rm safe}(z)) \geq - \alpha h(z) +  \tfrac{\| L_g h(z) \|^2}{4 \gamma \rho}. 
\]
If $\lambda > \alpha$, for closed-loop FoM \eqref{eq:closedloop} the set:
\begin{equation}
 S_\chi=\{x\in X:(\lambda-\alpha)h(\Pi(x))\ge\gamma V(x)\}
 \label{eq:final-safe}
\end{equation}
is forward invariant. Therefore, since $S_\chi\subseteq\mathcal C$, for all solutions $c \colon I \to X$ of $\dot{x} = \fcl(x)  = f_{\rm F}(x,\kappa_{\rm safe}(\Pi(x)))$:
\[
c(0) \in  S_\chi \qquad \implies \qquad h(\Pi(c(t))) \geq 0  \quad  \forall t \in I. 
\]  
\end{theorem}

\begin{proof}
Safety of $S_{\chi}$ follows from Corollary~\ref{cor:composition}. It only remains to verify the form of $S_{\chi}$ given in \eqref{eq:final-safe}. 
With the component-wise order on $R = \R^2_{(b,\ell)}$: 
\[
\begin{aligned}
S_{\chi} =\Eq^{\le}(\zero X,\chi) & = 
\{ x \in X : \chi(x) \geq 0 \}  \\
& = \{ x \in X : h(\Pi(x)) \geq 0, ~ H(x) \geq 0\}, 
\end{aligned}
\]
since $b = h \circ \Pi$. But $a = \lambda - \alpha > 0$ and $V \geq 0$, and therefore: 
\[
H(x)  \geq 0 ~  \implies ~ a h(\Pi(x)) \geq \gamma V(x) ~  \implies  ~ h(\Pi(x)) \geq 0. 
\]
This implies that $S_{\chi}$ is given by \eqref{eq:final-safe} and that $S_\chi\subseteq\mathcal C$. 
\end{proof}

Note that the barrier morphism in this case, $\chi$, takes values in $\R^2$. We can obtain a classical scalar barrier as well. 
To this end, consider the projection $\mathbf{pr}_{\ell} \colon \QT\to(\R,\sigma_\alpha)$, where $\sigma_\alpha(y)=-\alpha y$, is a strict dynamics-preserving morphism. Therefore, pasting with $\chi$ gives the scalar barrier $H=\pr_\ell\circ\chi$:
\begin{equation}
\begin{tikzcd}[column sep=12mm,row sep=8mm]
    X
    \ar[r,"\chi"]
    \ar[rr,bend left=30,"H"]
    \ar[d,"\fcl"']
    &
    \R^2
    \ar[r,"\pr_\ell"]
    \ar[d,"\sigma"]
    &
    \R
    \ar[d,"\sigma_\alpha"]
    \\
    T X 
    \ar[r,"T\chi"']
    \ar[rr,bend right=30,"TH"']
    \arrow[ur, phantom, sloped, "\geq", pos = 0.5]
    &
    T\R^2
    \ar[r,"T\pr_\ell"']
    &
    T\R.
\end{tikzcd}
~  \rightsquigarrow \quad \dot{H} \geq - \alpha H
\label{eq:H-pasting} \nonumber 
\end{equation}
This, again, implies that $S_{\chi} = \{ x \in X : H(x) \geq 0\}$ is forward invariant and the conclusions of Theorem \ref{thm:layered} follow. 

\section{Conclusion}

This paper presented a categorical framework in which to represent and analyze layered control architectures. To this end, we formulated ``control systems'' categorically via control coalgebras. While $\C = \Man$ was considered in this paper for simplicity, the framework is general and can be applied to a wide range of system types. 
Composition can be understood categorically as the pullback over a mediator. The main result showed how the universal property of the pullback yields CBMs on the composed system given CBMs on each subsystem. We applied this result to a layered architecture representing safety filtering a FoM with a RoM, recapturing previous results as a specialization. Looking forward, the goal is to apply this categorical framework to more complex layered architectures with heterogeneous control systems.

\newpage
\bibliographystyle{IEEEtran}
\bibliography{references}

\begin{thebibliography}{10}
\providecommand{\url}[1]{#1}
\csname url@samestyle\endcsname
\providecommand{\newblock}{\relax}
\providecommand{\bibinfo}[2]{#2}
\providecommand{\BIBentrySTDinterwordspacing}{\spaceskip=0pt\relax}
\providecommand{\BIBentryALTinterwordstretchfactor}{4}
\providecommand{\BIBentryALTinterwordspacing}{\spaceskip=\fontdimen2\font plus
\BIBentryALTinterwordstretchfactor\fontdimen3\font minus \fontdimen4\font\relax}
\providecommand{\BIBforeignlanguage}[2]{{%
\expandafter\ifx\csname l@#1\endcsname\relax
\typeout{** WARNING: IEEEtran.bst: No hyphenation pattern has been}%
\typeout{** loaded for the language `#1'. Using the pattern for}%
\typeout{** the default language instead.}%
\else
\language=\csname l@#1\endcsname
\fi
#2}}
\providecommand{\BIBdecl}{\relax}
\BIBdecl

\bibitem{Matni2024}
N.~Matni, A.~D. Ames, and J.~C. Doyle, ``A quantitative framework for layered multirate control: Toward a theory of control architecture,'' \emph{IEEE Control Systems Magazine}, vol.~44, no.~3, pp. 52--94, 2024.

\bibitem{cohen2024safety}
M.~H. Cohen, T.~G. Molnar, and A.~D. Ames, ``Safety-critical control for autonomous systems: Control barrier functions via reduced-order models,'' \emph{Annual Reviews in Control}, vol.~57, p. 100947, 2024.

\bibitem{Aoki1968Aggregation}
M.~Aoki, ``Control of large scale dynamic systems by aggregation,'' \emph{IEEE Transactions on Automatic Control}, vol.~13, pp. 246--253, Jun. 1968.

\bibitem{Mesarovic1970Hierarchical}
M.~D. Mesarovic, D.~Macko, and Y.~Takahara, \emph{Theory of Hierarchical, Multilevel, Systems}, ser. Mathematics in Science and Engineering.\hskip 1em plus 0.5em minus 0.4em\relax New York: Academic Press, 1970, vol.~68.

\bibitem{ZhongWonham1990Consistency}
H.~Zhong and W.~M. Wonham, ``On the consistency of hierarchical supervision in discrete-event systems,'' \emph{IEEE Transactions on Automatic Control}, vol.~35, no.~10, pp. 1125--1134, Oct. 1990.

\bibitem{Willems2007Behavioral}
J.~C. Willems, ``The behavioral approach to open and interconnected systems,'' \emph{IEEE Control Systems Magazine}, vol.~27, no.~6, pp. 46--99, Dec. 2007.

\bibitem{Tabuada2009Verification}
P.~Tabuada, \emph{Verification and Control of Hybrid Systems: A Symbolic Approach}.\hskip 1em plus 0.5em minus 0.4em\relax New York: Springer, 2009.

\bibitem{Mazo2024Contract}
M.~Mazo~Jr., W.~Compton, M.~H. Cohen, and A.~D. Ames, ``A contract theory for layered control architectures,'' 2024, arXiv preprint arXiv:2409.14902.

\bibitem{10.1145/3764587}
\BIBentryALTinterwordspacing
N.~V. Naik, A.~Pinto, and P.~Nuzzo, ``Contract embeddings for layered control architectures,'' \emph{ACM Trans. Embed. Comput. Syst.}, vol.~24, no.~5s, Sep. 2025. [Online]. Available: \url{https://doi.org/10.1145/3764587}
\BIBentrySTDinterwordspacing

\bibitem{Rosolia2022}
U.~Rosolia, A.~Singletary, and A.~D. Ames, ``Unified multirate control: From low-level actuation to high-level planning,'' \emph{IEEE Transactions on Automatic Control}, vol.~67, no.~12, pp. 6627--6640, 2022.

\bibitem{AmesCBF2017}
A.~D. Ames, X.~Xu, J.~W. Grizzle, and P.~Tabuada, ``Control barrier function based quadratic programs for safety critical systems,'' \emph{IEEE Transactions on Automatic Control}, vol.~62, no.~8, pp. 3861--3876, 2017.

\bibitem{Molnar2022}
T.~G. Moln{\'a}r, R.~K. Cosner, A.~W. Singletary, W.~Ubellacker, and A.~D. Ames, ``Model-free safety-critical control for robotic systems,'' \emph{IEEE Robotics and Automation Letters}, vol.~7, no.~2, pp. 944--951, 2022.

\bibitem{Compton2025}
W.~D. Compton, M.~H. Cohen, and A.~D. Ames, ``Learning for layered safety-critical control with predictive control barrier functions,'' in \emph{Proceedings of the 7th Annual Learning for Dynamics \& Control Conference}, ser. Proceedings of Machine Learning Research, vol. 283, 2025, pp. 153--165.

\bibitem{Rutten2000}
J.~J. M.~M. Rutten, ``Universal coalgebra: A theory of systems,'' \emph{Theoretical Computer Science}, vol. 249, no.~1, pp. 3--80, 2000.

\bibitem{AmesNESC2006}
A.~D. Ames, A.~Sangiovanni-Vincentelli, and S.~Sastry, ``Homogeneous semantics preserving deployments of heterogeneous networks of embedded systems,'' in \emph{Networked Embedded Sensing and Control}.\hskip 1em plus 0.5em minus 0.4em\relax Springer, 2006, pp. 127--154.

\bibitem{Tabuada2008}
P.~Tabuada, ``Controller synthesis for bisimulation equivalence,'' \emph{Systems \& Control Letters}, vol.~57, no.~6, pp. 443--452, 2008.

\bibitem{Vagner2015}
D.~Vagner, D.~I. Spivak, and E.~Lerman, ``Algebras of open dynamical systems on the operad of wiring diagrams,'' \emph{Theory and Applications of Categories}, vol.~30, no.~51, pp. 1793--1822, 2015.

\bibitem{Culbertson2020}
J.~Culbertson, P.~Gustafson, D.~E. Koditschek, and P.~F. Stiller, ``Formal composition of hybrid systems,'' \emph{Theory and Applications of Categories}, vol.~35, no.~45, pp. 1634--1682, 2020.

\bibitem{CapucciMyers2026}
M.~Capucci and D.~J. Myers, ``Compositionality of {Lyapunov} functions via assume-guarantee reasoning,'' 2026, arXiv:2604.03017.

\bibitem{CLT1}
A.~D. Ames, J.~Moeller, and P.~Tabuada, ``Categorical {Lyapunov} theory {I}: Stability of flows,'' \emph{Applied Categorical Structures}, vol.~34, no.~60, 2026.

\bibitem{CLT2}
A.~D. Ames, S.~Mattenet, and J.~Moeller, ``Categorical {Lyapunov} theory {II}: Stability of systems,'' 2025, arXiv:2505.22968.

\bibitem{cohen2025safety}
M.~H. Cohen, N.~Csomay-Shanklin, W.~D. Compton, T.~G. Molnar, and A.~D. Ames, ``Safety-critical controller synthesis with reduced-order models,'' in \emph{2025 American Control Conference (ACC)}.\hskip 1em plus 0.5em minus 0.4em\relax IEEE, 2025, pp. 5216--5221.

\bibitem{nanayakkara2026safetystateuncertaintyrobustifying}
\BIBentryALTinterwordspacing
R.~Nanayakkara, A.~D. Ames, and P.~Tabuada, ``Safety under state uncertainty: Robustifying control barrier functions,'' 2026. [Online]. Available: \url{https://arxiv.org/abs/2508.17226}
\BIBentrySTDinterwordspacing

\bibitem{moeller2026hybrid}
J.~Moeller and A.~D. Ames, ``Hybrid systems as coalgebras: Lyapunov morphisms for {Zeno} stability,'' \emph{arXiv preprint arXiv:2604.05255}, 2026.

\bibitem{NonlinearSystems}
H.~K. Khalil, \emph{Nonlinear Systems}, 3rd~ed.\hskip 1em plus 0.5em minus 0.4em\relax Upper Saddle River, NJ: Prentice Hall, 2002.

\bibitem{cohen2023characterizing}
M.~H. Cohen, P.~Ong, G.~Bahati, and A.~D. Ames, ``Characterizing smooth safety filters via the implicit function theorem,'' \emph{IEEE Control Systems Letters}, vol.~7, pp. 3890--3895, 2023.

\end{thebibliography}
\end{document}